\documentclass[12pt,reqno]{amsart}
\usepackage{amsmath}
\usepackage{amsthm}
\usepackage{amssymb}
\usepackage{graphics}
\usepackage{latexsym}

\numberwithin{equation}{section}
\newtheorem{thm}{Theorem}[section]
\newtheorem{prop}[thm]{Proposition}
\newtheorem{lem}[thm]{Lemma}

\theoremstyle{remark}

\newcommand{\p}{\partial}

\title[GWP for the Kawahara equation]{
}

\author[T. K. Kato]{
}
\email[Takamori Kato]{tk-kato@math.kyoto-u.ac.jp}


\begin{document}

\begin{center}
{\bf GLOBAL WELL-POSEDNESS FOR THE KAWAHARA EQUATION \\
 WITH LOW REGULARITY DATA}

\bigskip {\sc Takamori Kato}

\smallskip {\small Department of Mathematics, Kyoto University

Kyoto, 606-8502, Japan}

\end{center}
\begin{abstract}
We consider the global well-posedness for the Cauchy problem of the Kawahara equation which is one of fifth order KdV type equations. 
We first establish the local well-posedness in a more suitable function space for the global well-posedness by a variant of the Fourier restriction norm method. 
Next, we extend this local solution globally in time by the I-method. 
In the present paper, we apply the I-method to the modified Bourgain space. 
\end{abstract}
\maketitle

\section{Introduction}

We consider the global well-posedness (GWP) for the Cauchy problem of the Kawahara equation which is one of fifth order 
KdV type equations. 
\begin{align} \label{Ka}
\begin{cases}
& \p_t u + \alpha \p_x^5 u + \beta \p_x^3 u + \gamma \p_x (u^2)=0, 
\hspace{0.3cm} (t,x) \in [0,T] \times \mathbb{R}, \\
& u(0, \cdot)= u_0(\cdot) \in H^s (\mathbb{R}), 
\end{cases}
\end{align}
where $\alpha, \beta, \gamma \in \mathbb{R}$ with $\alpha, \gamma \neq 0$ and 
an unknown function $u$ is real valued. By the renormalization of $u$, we may assume that 
$\alpha =-1$, $\gamma=1$ and $\beta=-1$, $0$ or $1$. 
The Kawahara equation models the 
capillary waves on a shallow layer and the magneto-sound propagation in plasma (e.g. \cite{Ka}). 
Moreover this equation has solitary waves in the case $\beta=1$ and many conserved quantities 
(e.g. $L^2$, $H^2$, $\cdots$ ). 
Our aim is to prove GWP for (\ref{Ka}) with low regularity data. 
We first establish the local well-posedness (LWP) by using the Fourier restriction norm method introduced by Bourgain
 \cite{Bo}. Next, we extend these local solutions to global ones by the I-method. This method is developed by 
Colliander, Keel, Staffilani, Takaoka and Tao \cite{CoKe}, \cite{I02}. 

We briefly recall the local well-posedness results for (\ref{Ka}). 
Cui, Deng and Tao \cite{CDT} proved LWP in $H^s$ for $s>-1$, which was based on Kenig, Ponce and Vega's work \cite{KPV93}. 
Wang, Cui and Deng \cite{WCD} refined their argument to show LWP in $H^s$ for $s \geq -7/5$. 
Chen, Li, Miao and Wu \cite{CLMW} proved LWP in $H^s$ for $s >-7/4$, following the $[k, Z]$-multiplier norm method exploited by Tao \cite{Ta}. 
Chen and Guo \cite{CG} later showed LWP in $H^s$ for $s=-7/4$ by using the $\bar{F}^s$ type function space defined below. 
Following an idea of Bejenaru and Tao \cite{BT} and Kishimoto and Tsugawa \cite{KT}, 
we improved the previous results to have LWP in $H^s$ for $s \geq -2$ in \cite{TK}. 

\begin{thm} \label{thm_LWP}
Let $s \geq -2$. Then (\ref{Ka}) is locally well-posed in $H^s$. 
\end{thm}

On the other hand, when $s<-2$, we obtained ill-posedness in the following sense.

\begin{thm} \label{thm_ill}
Let $s <-2$. There exist $T>0$, $C_0>0$ and the sequence of initial data $\{ \phi_{N} \}_{N=1}^{\infty} \in H^{\infty} (\mathbb{R})$ 
satisfying the following conditions, for any $t \in (0,T]$, 
\begin{align*}
\text{(i)}~\| \phi_{N} \|_{H^s} \rightarrow 0 ~~\text{as} ~~ N  \rightarrow \infty, \hspace{0.5cm}
\text{(ii)} ~\| u_{N} (t) \|_{H^s} \geq C_0,
\end{align*}
where $u_N$ is the solution to (\ref{Ka}) with initial data $\phi_{N}$ obtained in Theorem~\ref{thm_LWP}. 
\end{thm} 
For the proof of Theorem~\ref{thm_LWP}, see \cite{TK}. This proof is based on Bejenaru and Tao's work \cite{BT}. 
These theorems imply that the critical exponent $s$ is equal to $-2$. 
Next, we extend the local solution obtained above globally in time by the I-method. 
Yan and Li \cite{YL} proved GWP in $H^s$ for $s >-63/58$, following the argument of \cite{CoKe}. 
Chen and Guo \cite{CG} used the argument of \cite{CoKeSt} to show GWP in $H^s$ for $s \geq -7/4$.
 We apply the I-method under the weaker regularity condition on data to obtain 
the global solution of (\ref{Ka}) for $s \geq -38/21$. 
But the function space used in \cite{TK} is not applicable for the proof of the global existence. In this paper, 
we adjoint the function space for the global well-posedness and so we reproduce the proof of LWP in the 
adjusted function space (see Section 3). 
The main result in this paper is the following.  

\begin{thm} \label{thm_main}
Let $s \geq -38/21$. Then, for any $T>0$, there exists a unique solution to (\ref{Ka}) in the function space 
$W^{s} ([0,T])$ defined below. Moreover, the data-to-solution map, 
$H^s \ni u_0 \mapsto u \in W^{s}([0,T])$, is locally Lipschitz continuous and 
\begin{align*}
\sup_{-T \leq t \leq T} \| u (t) \|_{H^s} \leq C T^{7/5 (2s+5)} \| u_0 \|_{H^s},
\end{align*} 
for some constant $C>0$. 
\end{thm}

 \vspace{0.3em}
 
 We now use the scaling argument. For $\lambda \geq 1$,  
\begin{align*}
u_{\lambda} (t,x):= \lambda^{-4} u(\lambda^{-5}t, \lambda^{-1} x), ~~~~
u_{0, \lambda}(x):=\lambda^{-4} u_0(\lambda^{-1} x). 
\end{align*}
If $u$ solves (\ref{Ka}), then $u_{\lambda}$ satisfies the following equation. 
\begin{align} \label{Ka2}
\begin{cases}
& \p_t u_{\lambda} - \p_x^5 u_{\lambda} + \beta  \lambda^{-2} \p_x^3 u_{\lambda} + \p_x (u_{\lambda}^2)=0, 
\hspace{0.3cm} (t,x) \in [0,\lambda^5 T] \times \mathbb{R}, \\
& u_{\lambda}(0, \cdot)= u_{0,\lambda}(\cdot) \in H^s (\mathbb{R}). 
\end{cases}
\end{align}
A direct calculation shows 
\begin{align*}
\| u_{0, \lambda} \|_{H^s} \leq \lambda^{-s-7/2} \| u_0 \|_{H^s}~~\text{for}~~ s<0.
\end{align*}
Therefore we can assume smallness of initial data when $s>-7/2$.
So we solves (\ref{Ka2}) for sufficiently small data. 

We first recall our local well-posedness result. 
The main idea is how to define the function space 
to construct local solutions. 
Here the Bourgian space $X^{s,b}$ plays an important role when $s$ is small. The Bourgain space 
$X^{s,b}$ is introduced by Bougain \cite{Bo} and equipped with the norm, 
\begin{align*}
\| u \|_{X^{s,b}}:= \| \langle \xi \rangle^s \langle \tau-p_{\lambda}(\xi) \rangle^{b} \widehat{u} \|_{L_{\tau,\xi}^2},
\end{align*}
where $p_{\lambda} (\xi)=\xi^5+ \beta \lambda^{-2} \xi^3$ and 
$\widehat{u}$ is the Fourier transform of $u$. 
The key is to establish the bilinear estimate of the nonlinearity $\p_x(u^2)$ as follows;
\begin{align} \label{BE}
\| \Lambda^{-1} \p_x (u v ) \|_{X^{s,b}} \leq C \| u \|_{X^{s,b}} \| v \|_{X^{s,b}},
\end{align}
where $\Lambda^b$ is the Fourier multiplier defined as 
$\Lambda^b:= \mathcal{F}_{\tau, \xi}^{-1} \langle \tau-p_{\lambda}(\xi) \rangle^b \mathcal{F}_{t, x}$ for $b \in \mathbb{R}$. 
Combining (\ref{BE}) and some linear estimates, the Fourier restriction norm method 
works to obtain LWP. Chen, Li, Miao and Wu \cite{CLMW} proved (\ref{BE}) in $X^{s,1/2+\varepsilon}$ with $0 < \varepsilon \ll 1$ when $s>-7/4$. 
This result was improved to $s=-7/4$ by Chen and Guo \cite{CG}. 
However, (\ref{BE}) fails for any $b \in \mathbb{R}$ when $s<-7/4$. 
To overcome this difficulty, we modify the Bourgain space $X^{s,b}$ to control strong nonlinear interactions and 
establish (\ref{BE}) for $s <-7/4$. The idea of the modification of $X^{s,b}$ was introduced by Bejenaru and Tao \cite{BT}. 
Remark that there is no general framework for modifying $X^{s,b}$. This is one of the most difficult points in our study. 
Following the similar argument to \cite{KT}, we obtained (\ref{BE}) in the critical case $s=-2$ in \cite{TK}. 
We now mention how to modify the Bourgain space. 
From the counterexamples of (\ref{BE}) (see appendix in \cite{TK}), we find the domain in which 
strong nonlinear interactions appear and make a suitable modification in this region. 
We first divide $\mathbb{R}^2$ into three parts as follows;
\begin{align*}
D_1:=& \bigl\{ (\tau,\xi) \in \mathbb{R}^2~;~ |\tau -p_{\lambda} (\xi)| \leq \frac{31}{32} |\xi|^5+ \frac{7}{8} \beta \lambda^{-2} |\xi|^3
~~\text{and}~~|\xi| \geq 1  \bigr\}, \\
D_2:=& \bigl\{ (\tau,\xi) \in \mathbb{R}^2~;~ |\tau -p_{\lambda} (\xi)| \geq \frac{31}{32} |\xi|^5+ \frac{7}{8} \beta \lambda^{-2} |\xi|^3
~~\text{and}~~|\xi| \geq 1  \bigr\}, \\
D_3:=& \bigl\{ (\tau,\xi) \in \mathbb{R}^2~;~|\xi| \leq 1 \bigr\}.
\end{align*} 
From the counterexamples of (\ref{BE}), the necessary conditions are 
\begin{align} \label{co_1}
& b \leq \frac{4}{5} s +\frac{19}{10} ~~\text{in}~~D_3, \\  
\label{co_2}
& b \leq \frac{s}{2} + \frac{11}{8} ~~\text{in}~~D_2.
\end{align}
Remark that these conditions only come from the $high \times high \rightarrow low$ interaction 
which means that a low frequency band is generated 
from an interaction between high and high frequencies bands. 
So the way to modify the Bourgain space is not strictly restricted. 
From (\ref{co_1}), we have to take $b \leq 3/10$ in $D_3$ when $s=-2$. 
Following the above argument, the function space $Z^s$ is equipped with the norm,
\begin{align*}
\| u \|_{Z^s}:= \| P_{D_1} u \|_{X_{(2,1)}^{s,1/2}}+ \| P_{D_2 \cup D_3} u \|_{X_{(2,1)}^{s+1,3/10}}
+ \| u \|_{Y^s},
\end{align*} 
where $P_{\Omega}$ is the Fourier projection onto a set $\Omega$ and 
$\| u \|_{Y^s} := \| \langle \xi \rangle^s \widehat{u} \|_{L_{\xi}^2 L_{\tau}^1 }$.  
Here $X_{(2,1)}^{s,b}$ is the Besov type Bourgain space defined by the norm,
\begin{align*}
\| u \|_{X_{(2,1)}^{s,b}}:=\Bigl\| \bigl\{ \| \chi_{A_j \cap B_k } \langle \xi \rangle^{s} \langle \tau-p_{\lambda}(\xi) \rangle^b \widehat{u} 
\|_{L_{\tau,\xi}^2} \bigr\}_{j.k \geq 0} 
 \Bigr\|_{l_j^2 l_k^1 },
\end{align*}
where $\chi_{\Omega}$ is the characteristic function of a set $\Omega$ and $A_j$, $B_k$ are dyadic decompositions as follows;
\begin{align*}
A_j:=& \bigl\{ (\tau,\xi) \in \mathbb{R}^2~;~2^j \leq \langle \xi \rangle < 2^{j+1} \bigr\}, \\
B_k:= & \bigl\{  (\tau,\xi) \in \mathbb{R}^2~;~2^k \leq \langle \tau-p_{\lambda}(\xi) \rangle < 2^{k+1} \bigr\}, 
\end{align*}
for $j,k \geq 0$. 
Using the function space $Z^s$, we obtain (\ref{BE}) for $s \geq -2$. 
Then the standard argument of the Fourier restriction norm method works to have LWP in $Z^s([0,T])$ for $s \geq -2$. 
Here $Z^s(I)$, for a time interval $I$, was defined by the norm, 
\begin{align*}
\| u \|_{Z^s(I)} := \inf \bigl\{ \| v \|_{Z^s}~;~u(t)=v(t), ~~\text{on}~~t \in I \bigr\}.
\end{align*}
For the details of the proof, see \cite{TK}. Note that the function space constructed in \cite{TK} is slightly 
different from $Z^s$ but essentially same as this one. 
 
\vspace{0.3em}

Next, we extend the local solution obtained above globally in time. 
But we have no conservation laws when $s$ is negative. 
To avoid this difficulty, we apply the I-method exploited by Colliander, Keel, Staffilani, Takaoka and Tao 
\cite{CoKe}, \cite{I02}. The main idea is to use a modified energy defined for less regular functions, which is not conserved. 
We now define the modified energy $E_I^{(2)} (u)$. The operator $I: H^s \rightarrow 
L^2$ is the Fourier multiplier satisfying $I= \mathcal{F}_{\xi}^{-1} m (\xi) \mathcal{F}_x$. Here $m(\xi) $ 
is a smooth and monotone function such that 
\begin{align*}
m( \xi )=
\begin{cases}
1 ~~& \text{for}~~ |\xi| \leq N, \\
|\xi|^s N^{-s}~~ & \text{for} ~~|\xi| \geq 2N,
\end{cases}
\end{align*}
for $s<0$ and large $N$. 
The functional $E_I^{(2)} (u)$ is defined as $E_I^{(2)} (u) (t):= \| I u(t) \|_{L^2}^2$. 
If we can control the growth of the modified energy $E_I^{(2)} (u)$ 
in time, this allows to iterate the local theory to continue the solution to any time $T$. 
In the I-method, the key estimate is the almost conservation law which implies that 
the increment of the modified energy is sufficiently small for small time interval and large $N$. 
Yan and Li \cite{YL} proved the almost conservation law for $E_I^{(2)} (u)$ and obtained 
GWP for $s >-63/58$, following \cite{CoKe}. Colliander, Keel, Staffilani, Takaoka and Tao \cite{CoKeSt} 
defined the new modified energy $E_I^{(4)} (u)$ by adding two suitable correction terms to $E_I^{(2)} (u)$ 
in order to remove some oscillations in that functional. They proved GWP for the KdV equation 
in $H^s$ when $s>-3/4$ by using this functional. Kishimoto \cite{Ki} slightly but essentially modified the Bourgain space 
to establish GWP for $s \geq -3/4$ 
(see also \cite{Gu}). 
Chen and Guo \cite{CG} defined the modified energy 
$E_I^{(4)} (u)$ for the Kawahara equation to obtain the almost conservation law and GWP for $s \geq -7/4$ 
in the following function space $\bar{F}^s$ introduced by Guo \cite{Gu}.
\begin{align*}
\| u \|_{\bar{F}^s}:=\| P_{\{ |\xi| \geq 1 \} } u \|_{X_{(2,1)}^{s,1/2}} + \| P_{\{ |\xi| \leq 1 \}} u \|_{L_x^2 L_t^{\infty}}.
\end{align*}  
Note that we obtain the same result as above if this function space is replaced by $X_{(2,1)}^{s,1/2}$. 
We encounter some difficulty to establish the almost conservation law because 
the Kawahara equation has less symmetries than the KdV equation. 
Our purpose is to establish both (\ref{BE}) and the almost conservation law when $s<-7/4$. 
In fact, the use of $X_{(2,1)}^{s,1/2} $ enables us to show that the almost conservation law for $E_I^{(4)} (u)$ holds
 for $s \geq -37/20$. 
However, (\ref{BE}) in $X_{(2,1)}^{s,1/2}$ breaks down for $s<-7/4$. 
Namely, we would not construct the local solution by the iteration argument. 
From the necessary condition (\ref{co_1}), we need to take $b_1 <1/2$ when $s < -7/4$ 
when the norm in $D_3$ is defined as $\| \cdot \|_{X_{(2,1)}^{s,b_1}}$. 
In the case $b_1 \geq 1/2$, we can recover two derivatives by the smoothing effects (see \cite{KPV91}). 
On the other hand, when $b_1 <1/2$, it is hard to establish the almost conservation law for $s<-7/4$ 
since the smoothing effect is weaker. Using the function space $Z^s$ defined above, we obtain (\ref{BE}) for 
$s \geq -2$ but the almost conservation law for $s \geq -6/5$ since $b_1=3/10$ in the function space $Z^s$. 
So we simultaneously need to control strong nonlinear interactions which come from (\ref{BE}) and 
the almost  conservation law. 
To overcome this difficulty, we establish the improved bilinear estimate 
which is sharp in some sense. 
Roughly speaking, this estimate implies that we can gain $4 b_1$ derivatives by using $X_{(2,1)}^{0,b_1}$. 
Following the improved bilinear estimate and the Sobolev inequality, we obtain the almost 
conservation law for $s \geq -4b_1$. From this and the necessary condition (\ref{co_1}), the minimum of $s$ is 
$-38/21$ with $b_1=19/42$. Following the above argument, we define the function space $W^s$ as follows;
\begin{align*}
\| u \|_{W^s}:= & \| P_{D_1} u \|_{X_{(2,1)}^{s,1/2}} + \| P_{D_2} u \|_{X_{(2,1)}^{s+s_2,b_2}}  \\
& \hspace{0.8cm} +\| P_{D_3} u \|_{X_{(2,1)}^{s,19/42}}+ \| u \|_{Y^s},
\end{align*}
where $s_2=25/168$ and $b_2=79/168$ which come from the necessary condition (\ref{co_2}). 
By using the function space $W^s$, we obtain both (\ref{BE}) and the almost conservation law for $-38/21 \leq s <0$. 
Remark we need another approach to prove the bilinear estimate because the function space used in \cite{TK} is different from 
the adjusted function space $W^s$ for the global well-posedness. 
Then we give the proof of the bilinear estimate in $W^s$, following Bejenaru and Tao \cite{BT}. 
For the details, see Proposition~\ref{prop_BE1} in Section 3. 
Moreover, we note that the difference between the almost conserved quantity $E_I^{(4)} (u)$ and the original modified energy $E_I^{(2)} (u)$ 
can be controlled by $E_I^{(2)} (u)$ when the time is fixed. For the details, see Proposition~\ref{prop_FTD} in section 4. 
From these estimates, the I-method is applicable to the modified Bourgain space $W^s$ so that we establish GWP for $s \geq -38/21$.

Remark, in the bilinear estimate for the Kawahara equation, 
we control only one nonlinear interaction. On the other hand, in the case of the KdV equation, 
we need to control three type nonlinear interactions at once. 
Therefore the way to modify the Bourgain space is strongly restricted. 
So it is difficult to apply the I-method to the modified Bourgain space for the KdV equation.

\vspace{0.5em}

We use the following notations throughout the present paper. $A \lesssim B$ means $A \leq C B$ for a positive constant $C$ and 
$A \sim B$ denotes both $A \lesssim B$ and $B \lesssim A$. 
Moreover $c+$ means $c+\varepsilon$, while $c-$ means $c-\varepsilon$, where $\varepsilon >0$ 
is enough small. The rest of this paper is planned as follows. In Section 2, we prepare the lemmas to prove 
the main results. In Section 3, we give the proof of the bilinear estimate and show LWP by the Fourier restriction norm method. 
In Section 4, we apply the I-method to modified Bougain space to show GWP. 

\vspace{0.5em}

\noindent
{\bf Acknowledgment.} The author would like to appreciate his supervisor Professor Yoshio Tsutsumi for 
many helpful conversation and encouragement and thank Professor Kotaro Tsugawa and Professor Nobu 
Kishimoto for helpful comments.

\section{Preliminaries}

In this section, we prepare important lemmas to show the main estimates. 
When we use the variables $(\tau,\xi)$, $(\tau_1,\xi_1)$ and $(\tau_2, \xi_2)$, 
we always assume the relation as follows. 
\begin{align*}
(\tau,\xi)=(\tau_1,\xi_1)+ (\tau_2,\xi_2).
\end{align*} 
For a normed space $\mathcal{X}$ and a set $\Omega$, $\| \cdot \|_{\mathcal{X} (\Omega)} $ denotes 
$\| f \|_{\mathcal{X} (\Omega) }:= \| \chi_{\Omega} f \|_{\mathcal{X}}$ where $\chi_{\Omega}$ is the 
characteristic function of $\Omega$. 
For dyadic numbers $N, M \geq 1$, $A_N$ and $B_M$ are dyadic decompositions defined as 
\begin{align*}
A_N:= \bigl\{ (\tau,\xi) ~; ~ N/2 \leq \langle \xi \rangle \leq  2N \bigr\},
\hspace{0.3cm} 
B_M:= \bigl\{ (\tau,\xi) ~ ;~M/2 \leq \langle \tau-p_{\lambda} (\xi) \rangle \leq 2M \bigr\},
\end{align*}
for dyadic numbers $N,M \geq 1$. 

The next two lemmas play a crucial role to establish the bilinear estimate. 
\begin{lem} \label{lem_L_4_1}
Suppose that each $\widehat{u}$ and $\widehat{v}$ is restricted to $A_N$ for a dyadic number $N \geq 1$. 
If $b+b' \geq 7/8$ and $b, b' >3/8$, then 
\begin{align} \label{es_L_4_11}
\| |\xi|^{3/4} \widehat{u} * \widehat{v} \|_{L_{\tau,\xi}^2(|\xi| \geq 1)} 
\lesssim \| u \|_{X_{(2,1)}^{0,b}} \| v \|_{X_{(2,1)}^{0,b'}}.
\end{align}  
Moreover, 
\begin{align*}
K= \inf \bigl\{ |\xi_1-\xi_2| ~; ~\tau_1, \tau_2~ \text{s.t}~ (\tau_1,\xi_1) \in \text{supp}~\widehat{u},~ (\tau_2,\xi_2) \in 
\text{supp}~\widehat{v}   \bigr\}>0,
\end{align*}
then we have
\begin{align} \label{es_L_4_12} 
\| |\xi|^{1/2} \widehat{u}* \widehat{v} \|_{L_{\tau,\xi}^2} 
\lesssim K^{-3/2} \| u \|_{X_{(2,1)}^{0,1/2}} \| v \|_{X_{(2,1)}^{0,1/2}}.
\end{align}
\end{lem}

\begin{lem} \label{lem_L_4_2}
Assume that $\widehat{v}$ is supported on $A_N$ for $N \geq 1$ and $\widehat{u}$ is an arbitrary test function. 
If $b+b' \geq 7/8$ and $b,b' >3/8$, then 
\begin{align} \label{es_L_4_21}
\| P_{ \{ \langle \tau-p_{\lambda} (\xi) \rangle \sim M_0 \}} u v \|_{L_{t,x}^2}
\lesssim M_0^{b} \| |\xi|^{-3/4} \widehat{u} \|_{L_{\tau,\xi}^2 (|\xi| \geq 1)}
\| v \|_{X_{(2,1)}^{0,b'}}.
\end{align}
Moreover, there exists a non-empty set $\Omega \subset \mathbb{R}^2$ such that
\begin{align*}
K= \inf \bigl\{ |\xi+\xi_2|~; ~\tau, \tau_2 ~\text{s.t.} ~(\tau, \xi) \in \Omega,~(\tau_2,\xi_2) \in 
\text{supp}~ \widehat{v} \bigr\}>0,
\end{align*}
then we have 
\begin{align} \label{es_L_4_22}
\| P_{ \{ \langle \tau-p_{\lambda} (\xi) \rangle \sim M_0 \}} u v \|_{L_{t,x}^2}
\lesssim K^{-3/2}~ M_0^{1/2} \| |\xi|^{-1/2} \widehat{u} \|_{L_{\tau,\xi}^2} 
\| v \|_{X_{(2,1)}^{0,1/2}}. 
\end{align}
\end{lem}
For the proofs of these lemmas, see \cite{TK1}.

We put a one parameter semigroup $U_{\lambda} (t)$ defined by 
\begin{align*}
U_{\lambda} (t) := \mathcal{F}_{\xi}^{-1} \exp (i p_{\lambda} (\xi) t) \mathcal{F}_x.
\end{align*}
From the definition, $W^s([0,T])$ has the following property. 
\begin{align*}
X^{s,1/2+} ([0,T]) \hookrightarrow W^s ([0,T]) \hookrightarrow C([0,T]; H^s),
\end{align*}
which implies the following linear estimates. 

\begin{prop} \label{prop_linear1}
Let $s \in \mathbb{R}$, $\lambda \geq 1$ and $T>0 $. Then we have
\begin{align*}
\| U_{\lambda} (t) u_0 \|_{W^s ([0,T])} \lesssim \| u_0 \|_{H^s}. 
\end{align*}
\end{prop}

\begin{prop} \label{prop_linear2}
Let $s \in \mathbb{R}$, $\lambda \geq 1$ and $T>0$. Then  we have, for any $t \in [0,T]$,
\begin{align*}
\| \int_0^t U_{\lambda} (t-s) F(s) ds \|_{W^s ([0,T])} \lesssim \| \Lambda^{-1} F \|_{W^s ([0,T])}. 
\end{align*}
\end{prop}

For the proofs of the above propositions, see \cite{BT}. 

\section{Local well-posedness}
In this section, we establish LWP in $W^s([0,T])$ for $s \geq -38/21$. 
The bilinear estimate in $W^{s}$ is stated as follows. 
\begin{prop} \label{prop_BE_W}
Let $s \geq -38/21$. Then the following estimate holds.
\begin{align} \label{BE_W}
\| \Lambda^{-1} \p_x (uv) \|_{W^s} \lesssim \| u \|_{W^s} \| v \|_{W^s}. 
\end{align}
\end{prop}
The proof of the bilinear estimate in \cite{TK} is based on the argument of Kenig, Ponce and Vega \cite{KPV96}. 
But this method is not applicable in the proof of the above bilinear estimate because the function space 
$W^s$ is the Besov type space. Then we use the similar argument to Bejenaru and Tao \cite{BT}. 
Note that $W^s$ has the $L_{\xi}^2$-property, namely, 
\begin{align*}
\| u \|_{W^s}^2 \sim \sum_{N \geq 1} \| P_{ \{ \langle \xi \rangle \sim N \} } u \|_{W^s}^2,
\end{align*}  
for dyadic numbers $N$. From this, we can reduce (\ref{BE_W}) to the following. 
\begin{prop} \label{prop_BE1}
Assume that $\widehat{u}$ and $\widehat{v}$ are restricted to $A_{N_1}$ and $A_{N_2}$ for 
dyadic numbers $N_1, N_2 \geq 1$. 
Then we have 
\begin{align} \label{red_BE}
\| P_{\{ \langle \xi \rangle \sim N_0  \} } \Lambda^{-1} \p_x (uv) \|_{W^{s}} \lesssim C(N_0, N_1, N_2) 
\| u \|_{W^s} \| v \|_{W^s},
\end{align}
for a dyadic number $N_0 \geq 1$ in the following six cases. 

\vspace{0.3em}
\noindent
(i) At least two of $ N_0, N_1, N_2$ are less than some universal constant and 
\\   $ C(N_0, N_1, N_2) \sim 1$ . \\ 
(ii) $N_0=1$, $N_1 \sim N_2 \gg 1$ and $C(N_0,N_1, N_2) \sim 1$ . \\
(iii) $N_0 >1$, $N_1 \sim N_2 \gg N_0$ and $C(N_0,N_1,N_2) \sim N_0^{- \delta}$ for some $\delta >0$. \\
(iv) $N_1=1$, $N_0 \sim N_2 \gg 1$ and $C(N_0,N_1, N_2) \sim 1$. \\
(v) $N_1>1$, $N_0 \sim N_2 \gg N_1$ and $C(N_0,N_1,N_2) \sim N_1^{-\delta}$ for some $\delta>0$. \\
(vi) $N_0 \sim N_1 \sim N_2 \gg 1$ and $C(N_0, N_1, N_2) \sim 1$.
\end{prop}

By using H\"{o}lder's and Young's inequalities and Lemmas~\ref{lem_L_4_1} and \ref{lem_L_4_2}, 
we obtain the above proposition.  

\begin{proof}[Proof of Proposition~\ref{prop_BE1}] 
We first see the properties of the function space $W^s$. 
From the definition, $\| u \|_{X_{(2,1)}^{s,19/42}}  \lesssim \| u \|_{W^s} \lesssim \| u \|_{X_{(2,1)}^{s,1/2}}$.   
From the Schwarz inequality, $\| u \|_{Y^s} \lesssim \| u \|_{X_{(2,1)}^{s,1/2}}$. 

\vspace{0.3em}

{\bf Estimate for (i).} In this case, all $N_0, N_1$ and $N_2 \lesssim 1$. We use the Young inequality to obtain
\begin{align*}
\sum_{M_0 \geq 1} M_0^{-1/2} \| ( \langle \xi \rangle^s \widehat{u} ) * (\langle \xi \rangle^s \widehat{v}) \|_{L_{\tau,\xi}^2(B_{M_0})} 
\lesssim \| \langle \xi \rangle^s \widehat{u} \|_{L_{\xi}^1 L_{\tau}^2} \| \langle \xi \rangle^s \widehat{v} \|_{L_{\xi}^2 L_{\tau}^1},
\end{align*} 
which is bounded by $\| u \|_{X^{s,0}} \| v \|_{Y^s} $ from the H\"{o}lder inequality. 
In the other cases, we often use the following algebraic relation. 
\begin{align} \label{alg}
 M_{max}:= & \max \bigl\{ |\tau-p_{\lambda} (\xi)|, | \tau_1-p_{\lambda} (\xi_1) | , |\tau_2-p_{\lambda}(\xi_2) | \bigr\} \nonumber \\
& \geq \frac{1}{3} | (\tau-p_{\lambda} (\xi)) -(\tau_1-p_{\lambda} (\xi_1)) -\{ (\tau-\tau_1)-p_{\lambda} (\xi -\xi_1) \} | \nonumber \\ 
& \geq  \frac{5}{6} \bigl| \xi \xi_1 (\xi-\xi_1) \big\{ \xi^2+ \xi_1^2 + (\xi-\xi_1)^2+ \frac{6}{5} \beta \lambda^{-2} \bigr\} \bigr|. 
\end{align}

{\bf Estimate for (ii).} In this case, $\widehat{u} * \widehat{v}$ is supported on $D_3$ and 
$M_{max} \gtrsim |\xi| N_1^4$. In the case $|\xi| \lesssim N_1^{-4} $, we use 
the H\"{o}lder inequality to obtain
\begin{align*}
& N_1^{-2s} \sum_{M_0 \geq 1} 
\| |\xi|~ ( \langle \xi \rangle^s \widehat{u} )* ( \langle \xi \rangle^s \widehat{v}) \|_{ L_{\tau,\xi}^2 (B_{M_0}) } \\
& \lesssim N_1^{-2s-6} \| (\langle \xi \rangle^s \widehat{u}) * (\langle \xi \rangle^s \widehat{v}) \|_{L_{\xi}^{\infty} L_{\tau}^2} ,
\end{align*}
which is bounded by $N_1^{-2s-6} \| u \|_{X^{s,0}} \| v \|_{Y^s}$ from the Young inequality. 
Therefore we assume $N_1^{-4} \lesssim |\xi| \leq 1$. 

\vspace{0.3em}

(Ia) Consider the case $\widehat{u}$ is supported on $D_2$. From the algebraic relation (\ref{alg}), 
either $|\tau-p_{\lambda} (\xi)| \sim |\tau_1-p_{\lambda} (\xi_1)| \gtrsim N_1^5$ or 
$| \tau_1-p_{\lambda} (\xi_1)| \sim |\tau_2-p_{\lambda} (\xi_2)| \gtrsim N_1^5 $. 
In the former case, we use Young's inequality to have
\begin{align*}
& N_1^{-2s} \sum_{ M_0 \gtrsim N_1^5}  M_0^{-1/2} 
\| |\xi| (\langle \xi \rangle^s \widehat{u}) 
* ( \langle \xi \rangle^s \widehat{v} ) \|_{L_{\tau,\xi}^2 (B_{M_0})} \\
& \lesssim N_1^{-2s-5/2} \| ( \langle \xi \rangle^s \widehat{u} ) * 
(\langle \xi \rangle^s \widehat{v}) \|_{ L_{\xi}^{\infty} L_{\tau}^2 } 
 \lesssim  N_1^{-2s-5/2} \| \langle \xi \rangle^s \widehat{u} \|_{L_{\tau,\xi}^2} 
\| \langle \xi \rangle^s \widehat{v} \|_{L_{\xi}^2 L_{\tau}^1},
\end{align*}
which is bounded by $N_1^{-2s-5} \| u \|_{X^{s+s_2,b_2}} \| v \|_{Y^s}$ 
with $s_2=25/168$ and $b_2= 79/168$ from the H\"{o}lder inequality. In the latter case, 
we may assume that $\widehat{v}$ is supported on $D_2$ and $|\tau-p_{\lambda} (\xi)| \lesssim N_1^5$. 
Combining the H\"{o}lder inequality and the Young inequality, we obtain  
\begin{align*}
& N_1^{-2s} \sum_{ 1 \leq M_0 \lesssim N_0^{5} }  M_0^{-23/42} 
\| |\xi|~ (\langle \xi \rangle^s \widehat{u}) 
* ( \langle \xi \rangle^s \widehat{v} ) \|_{L_{\tau,\xi}^2 (B_{M_0})} \\
& \lesssim N_1^{-2s} \| ( \langle \xi \rangle^s \widehat{u} ) * 
(\langle \xi \rangle^s \widehat{v}) \|_{ L_{\xi}^{\infty} L_{\tau}^{\infty} } 
 \lesssim  N_1^{-2s} \| \langle \xi \rangle^s \widehat{u} \|_{L_{\tau,\xi}^2} 
\| \langle \xi \rangle^s \widehat{v} \|_{L_{\xi,\tau}^2 }, \\
& \lesssim N_1^{-2s-5} \| u \|_{X^{s+s_2,b_2}} \| v \|_{X^{s+s_2,b_2}}.
\end{align*}
In the same manner as above, we obtain the desired estimate in the case $\widehat{v}$ is supported on $D_2$. 
Following the above estimates, we only prove (\ref{red_BE}) when both $\widehat{u}$ and $\widehat{v}$ are restricted to $D_1$ 
and $|\tau-p_{\lambda} (\xi)| \lesssim N_1^5 $. 

\vspace{0.3em}

(Ib) Consider the case $M_{max}=|\tau-p_{\lambda} (\xi)|$. We may assume $|\tau-p_{\lambda} (\xi)| \sim |\xi| N_1^4$. 
Firstly, we deal with the case $N_1^{-3/2} \lesssim |\xi| \leq 1$. 
From $|\tau-p_{\lambda} (\xi)| \gtrsim N_1^{5/2}$, we use (\ref{es_L_4_12}) with $K \sim N_1$ to obtain
\begin{align*}
& N_1^{-2s} \sum_{M_0 \gtrsim N_1^{5/2} } M_0^{-23/42}  \| |\xi|~ (\langle \xi \rangle^s \widehat{u}) * 
(\langle \xi \rangle^s \widehat{v}) \|_{L_{\tau,\xi}^2(B_{M_0}) } \\
& \lesssim 
N_1^{-2s-2} \sum_{M_0 \gtrsim N_1^{5/2}} M_0^{-1/21} \| |\xi|^{1/2} ~
(\langle \xi \rangle^s \widehat{u}) * (\langle \xi \rangle^s \widehat{v}) \|_{L_{\tau,\xi}^2 } \\
& \lesssim 
N_1^{-2s-76/21} \| u \|_{X_{(2,1)}^{s,1/2}} \| v \|_{X_{(2,1)}^{s,1/2}}.
\end{align*}   
Secondly, we consider the case $N_1^{-4} \lesssim |\xi| \lesssim N_1^{-3/2}$.  
From $|\tau-p_{\lambda} (\xi)| \lesssim N_1^{5/2}$, we use H\"{o}lder's and Young's inequalities to have 
\begin{align*}
& N_1^{-2s} \sum_{M_0 \lesssim N_1^{5/2} } M_0^{-23/42}  \| |\xi|~ (\langle \xi \rangle^s \widehat{u}) * 
(\langle \xi \rangle^s \widehat{v}) \|_{L_{\tau,\xi}^2(B_{M_0}) } \\
& \lesssim 
N_1^{-2s-6} \sum_{M_0 \gtrsim N_1^{5/2}} M_0^{20/21} \| 
(\langle \xi \rangle^s \widehat{u}) * (\langle \xi \rangle^s \widehat{v}) \|_{L_{\xi}^{\infty} L_{\tau}^2 } \\
& \lesssim 
N_1^{-2s-76/21} \| u \|_{X_{(2,1)}^{s,1/2}} \| v \|_{X_{(2,1)}^{s,1/2}}.
\end{align*}   
Finally, we estimate the $Y^s$ norm of $\Lambda^{-1} \p_x (u v)$. We use H\"{o}lder's inequality to have 
\begin{align*}
& N_1^{-2s} \| |\xi|~ \langle \tau-p_{\lambda} (\xi) \rangle^{-1}
(\langle \xi \rangle^s \widehat{u}) * (\langle \xi \rangle^s \widehat{v}) \|_{L_{\xi}^{2} L_{\tau}^1 } \\
& \hspace{0.8cm} \lesssim 
N_1^{-2s-4} \| (\langle \xi \rangle^s \widehat{u}) * (\langle \xi \rangle^s \widehat{v}) \|_{L_{\xi}^{\infty} L_{\tau}^1 } ,
\end{align*}
which is bounded by $N_1^{-2s-4} \| u \|_{Y^s} \| v \|_{Y^s}$ from Young's inequality. 

\vspace{0.3em}

(Ic) Consider the case $M_{max} =|\tau_1-p_{\lambda}( \xi_1)|$. From $|\tau_1-p_{\lambda} (\xi)| \gtrsim |\xi| N_1^4$, 
We use (\ref{es_L_4_22}) with $K \sim N_1$ to have 
\begin{align*}
& N_1^{-2s} \sum_{1 \leq M_0 \lesssim N_1^{5} } M_0^{-23/42} \| |\xi|~
(\langle \xi \rangle^s \widehat{u}) * (\langle \xi \rangle^s \widehat{v}) \|_{L_{\xi, \tau}^2(B_{M_0}) } \\
& \lesssim 
N_1^{-2s-2} \sum_{1 \leq M_0 \lesssim N_1^{5} } M_0^{-23/42} \| |\xi|^{1/2} 
(\langle \xi \rangle^s \langle \tau -p_{\lambda} (\xi) \rangle^{1/2} \widehat{u})* 
(\langle \xi \rangle^s \widehat{v} ) \|_{L_{\tau,\xi}^2(B_{M_0}) }   \\
& \lesssim 
N_1^{-2s-4} \| u \|_{X^{s,1/2}} \| v \|_{X_{(2,1)}^{s,1/2}}.  
\end{align*}
The case $M_{max}=|\tau_2 -p_{\lambda} (\xi_2)| $ is identical to the above case. 

{\bf Estimate for (iii)} In this case, $M_{max} \gtrsim N_0 N_1^4$ from (\ref{alg}).  

(IIa) Consider the case $\widehat{u}$ is restricted to $D_2$. From $N_1^5 \gg N_0 N_1^4$ and (\ref{alg}), 
either $|\tau-p_{\lambda} (\xi)| \sim |\tau_1-p_{\lambda} (\xi_1) | \gtrsim N_1^5$ or 
$|\tau_1-p_{\lambda} (\xi_1)| \sim |\tau_2-p_{\lambda} (\xi_2)| \gtrsim N_1^5$. 
In the former case, we use the H\"{o}lder inequality and the Young inequality to have 
\begin{align*}
& N_0^{s+1} N_1^{-2s} \sum_{M_0 \gtrsim N_1^{5} } M_0^{-1/2} \| 
(\langle \xi \rangle^s \widehat{u}) * (\langle \xi \rangle^s \widehat{v}) \|_{L_{\tau,\xi}^2(B_{M_0}) } \\
& \lesssim 
N_0^{s+3/2} N_1^{-2s-5/2}  \| 
(\langle \xi \rangle^s  \widehat{u})* (\langle \xi \rangle^s \widehat{v} ) \|_{L_{\xi}^{\infty} L_{\xi}^2}   
 \lesssim 
N_0^{s+3/2} N_1^{-2s-5/2} \| \langle \xi \rangle^s \widehat{u} \|_{L_{\tau,\xi}^2} \| v \|_{Y^s},  
\end{align*}
which is bounded by $N_0^{s+3/2} N_{1}^{-2s-5} \| u \|_{X^{s+s_2,b_2}} \| v \|_{Y^s} $ from the definition. 
In the latter case, we may assume that $|\tau-p_{\lambda} (\xi)| \lesssim N_1^5$ from the above estimate and 
$\widehat{v}$ is supported on $D_2$. Combining the H\"{o}lder inequality and the Young inequality, we have
\begin{align*}
& N_0^{s+1} N_1^{-2s} \sum_{ 1 \leq M_0 \lesssim N_1^{5} } M_0^{-1/2+} \| 
(\langle \xi \rangle^s \widehat{u}) * (\langle \xi \rangle^s \widehat{v}) \|_{L_{\tau,\xi}^2(B_{M_0}) } \\
& \lesssim 
N_0^{s+3/2} N_1^{-2s+}  \| 
(\langle \xi \rangle^s  \widehat{u})* (\langle \xi \rangle^s \widehat{v} ) \|_{L_{\tau, \xi}^{\infty}}   
 \lesssim 
N_0^{s+3/2} N_1^{-2s+} \| \langle \xi \rangle^s \widehat{u} \|_{L_{\tau,\xi}^2} \| \langle \xi \rangle^s \widehat{v} \|_{L_{\tau,\xi}^2},  
\end{align*}
which is bounded by $N_0^{s+3/2} N_1^{-2s-5+} \| u \|_{X^{s+s_2,b_2}} \| v \|_{X^{s+s_2, b_2}}$. 
The case $\widehat{v}$ is supported on $D_2$ is identical to the above case. 
Therefore  we only consider the case both $\widehat{u}$ and $\widehat{v}$ are restricted to $D_1$. 
 
 \vspace{0.3em}
 
 (IIb) Consider the case $M_{max}=|\tau-p_{\lambda} (\xi)|$. From the algebraic relation (\ref{alg}), 
 $\widehat{u} * \widehat{v} $ is supported on $D_2$. We use (\ref{es_L_4_12}) with $K \sim N_1 $ to obtain 
\begin{align*}
& N_0^{s+s_2+1} N_1^{-2s} \sum_{ M_0 \gtrsim N_0 N_1^{4} } M_0^{-1+b_2} \| 
(\langle \xi \rangle^s \widehat{u}) * (\langle \xi \rangle^s \widehat{v}) \|_{L_{\tau,\xi}^2(B_{M_0}) } \\
& \lesssim 
N_0^{s+s_2+b_2} N_1^{-2s-4+4b_2}  \| 
(\langle \xi \rangle^s  \widehat{u})* (\langle \xi \rangle^s \widehat{v} ) \|_{L_{\tau, \xi}^{2}}   
 \lesssim 
N_0^{s+5/42} N_1^{-2s-76/21} \| u \|_{X_{(2,1)}^{s,1/2}} \| v \|_{X_{(2,1)}^{s,1/2}}.  
\end{align*}
 Next, we estimate the norm $Y^s$ of $\Lambda^{-1} \p_x (uv)$. H\"{o}lder's and Young's inequalities imply that
\begin{align*}
 & N_0^{s+1} N_1^{-2s}  \| \langle \tau-p_{\lambda} (\xi) \rangle^{-1}
(\langle \xi \rangle^s \widehat{u}) * (\langle \xi \rangle^s \widehat{v}) \|_{L_{\xi}^2L_{\tau}^1 } \\
& \lesssim 
N_0^{s+1/2} N_1^{-2s-4}  \| 
(\langle \xi \rangle^s  \widehat{u})* (\langle \xi \rangle^s \widehat{v} ) \|_{L_{\xi}^{\infty} L_{\tau}^1}   
 \lesssim 
N_0^{s+1/2} N_1^{-2s-4} \| u \|_{Y^s} \| v \|_{Y^s}. 
\end{align*}
 
 (IIc) Consider the case $M_{max}=|\tau_1-p_{\lambda} (\xi_1)|$. From the above estimate, we may assume 
$|\tau-p_{\lambda}(\xi) | \lesssim N_1^5$. We use (\ref{es_L_4_22}) with $K \sim N_1$ to have
\begin{align*}
& N_0^{s+1} N_1^{-2s} \sum_{ 1 \leq M_0 \lesssim N_1^{5} } M_0^{-1/2+} \| 
(\langle \xi \rangle^s \widehat{u}) * (\langle \xi \rangle^s \widehat{v}) \|_{L_{\tau,\xi}^2(B_{M_0}) } \\
& \lesssim 
N_0^{s+1} N_1^{-2s-2+}  \| \langle \xi \rangle^s  \widehat{u} \|_{L_{\tau,\xi}^2} \| v \|_{X_{(2,1)}^{s,1/2}}   
 \lesssim 
N_0^{s+1/2} N_1^{-2s-4+} \| u \|_{X_{(2,1)}^{s,1/2}} \| v \|_{X_{(2,1)}^{s,1/2}},    
\end{align*}
which is an appropriate bound. In the same manner as above, we obtain the desired estimate in the case 
$M_{max}= |\tau_2-p_{\lambda} (\xi_2)|$ by symmetry. 

\vspace{0.3em}

{\bf Estimate for (iv).} In the case $|\xi_1| \lesssim N_0^{-4}$, we easily obtain the required estimate 
by H\"{o}lder's and Young's inequalities. So we only prove (\ref{red_BE}) when $N_0^{-4} \lesssim | \xi_1| 
\leq 1$. 

\vspace{0.3em}

(IIIa) Consider the case $M_{max} =|\tau_1-p_{\lambda} (\xi_1)|$. From 
$|\tau_1-p_{\lambda} (\xi_1) | \gtrsim |\xi_1| N_0^4$, we use the Young inequality to have
\begin{align*}
 N_0^{1}  \sum_{M_0 \geq 1} M_0^{-1/2} 
\| \widehat{u}*( \langle \xi \rangle^s \widehat{v}) \|_{L_{\tau,\xi}^2 (B_{M_0})} 
& \lesssim     
\| ( |\xi|^{-1/4} \langle \tau-p_{\lambda}(\xi) \rangle^{1/4} \widehat{u})*( \langle \xi \rangle^s \widehat{v}) \|_{L_{\tau,\xi}^2} \\ 
& \lesssim \| |\xi|^{-1/4} \langle \tau-p_{\lambda}(\xi) \rangle^{1/4} \widehat{u} \|_{L_{ \xi}^1 L_{\tau}^2 }
\| v \|_{Y^s},
\end{align*}
which is bounded by $\| u \|_{X^{s,1/4}} \| v \|_{Y^s}$ from the H\"{o}lder inequality. 

\vspace{0.3em}
 
(IIIb) Consider the case $M_{max}=| \tau-p_{\lambda} (\xi) |$. We use Young's inequality to have
\begin{align*}
& N_0^{1}  \sum_{M_0 \gtrsim |\xi_1| N_0^4 } M_0^{-1/2} 
\| \widehat{u}*( \langle \xi \rangle^s \widehat{v}) \|_{L_{\tau,\xi}^2 (B_{M_0})} \\
& \lesssim     
\| ( |\xi|^{-1/4} \widehat{u})*( \langle \xi \rangle^s \widehat{v}) \|_{L_{\tau,\xi}^2} 
\lesssim \| |\xi|^{-1/4} \widehat{u} \|_{L_{ \xi}^1 L_{\tau}^2 }
\| v \|_{Y^s},
\end{align*}
which implies the desired estimate from H\"{o}lder's inequality.

\vspace{0.3em}

(IIIc) Consider the case $M_{max}=| \tau_2-p_{\lambda} (\xi_2)|$. 
We use the Young inequality to have
\begin{align*}
& N_0^{1}  \sum_{M_0 \geq 1 } M_0^{-1/2} 
\| \widehat{u}*( \langle \xi \rangle^s \widehat{v}) \|_{L_{\tau,\xi}^2 (B_{M_0})} \\
& \lesssim     
\| ( |\xi|^{-1/4} \widehat{u})*( \langle \xi \rangle^s \langle \tau-p_{\lambda} (\xi) \rangle^{1/4} \widehat{v}) \|_{L_{\tau,\xi}^2} 
\lesssim \| |\xi|^{-1/4}  \widehat{u} \|_{L_{ \tau,\xi}^1 }
\| v \|_{X^{s,1/4}},
\end{align*}
which implies the required estimate.

\vspace{0.3em}

{\bf Estimate for (v).} From the algebraic relation (\ref{alg}), $M_{max} \gtrsim N_1 N_0^4$.

\vspace{0.3em} 

(IVa) Consider the case $\widehat{u} *\widehat{v}$ is supported on $D_2$. 
From (\ref{alg}), either $|\tau-p_{\lambda} (\xi)| \sim |\tau_1-p_{\lambda} (\xi_1)| \gtrsim N_0^5 $ 
or $|\tau-p_{\lambda} (\xi)| \sim |\tau_2-p_{\lambda} (\xi_2) | \gtrsim N_0^5$.
In the former case, $\widehat{u}$ is supported on $D_2$. We use the Young inequality to have
\begin{align*}
& N_0^{1} N_1^{-s}  \sum_{M_0 \gtrsim N_0^5} M_0^{-1/2} 
\| ( \langle \xi \rangle^s \widehat{u})*( \langle \xi \rangle^s \widehat{v}) \|_{L_{\tau,\xi}^2 (B_{M_0} )} \\
& \lesssim N_0^{-3/2} N_1^{-s}   
\| ( \langle \xi \rangle^s \widehat{u})*( \langle \xi \rangle^s \widehat{v}) \|_{L_{\tau,\xi}^2} 
\lesssim N_0^{-3/2} N_1^{-s} \| \langle \xi \rangle^s \widehat{u} \|_{L_{ \xi}^1 L_{\tau}^2 }
\|  \langle \xi \rangle^s \widehat{ v }\|_{L_{\xi}^2 L_{\tau}^1 } ,
\end{align*}
which is bounded by $ N_1^{-s-7/2} \|  u \|_{X^{s+s_2,b_2}} \| v \|_{Y^s} $ from 
the H\"{o}lder inequality. 
The latter case is almost identical to the above case. 

\vspace{0.3em}

(IVb) Consider the case $\widehat{v}$ is supported on $D_2$. 
From the algebraic relation (\ref{alg}), either $|\tau_2-p_{\lambda} (\xi_2)| \sim |\tau-p_{\lambda} (\xi)| \gtrsim N_0^5 $
 or $|\tau_2-p_{\lambda} (\xi_2)| \sim |\tau_1-p_{\lambda} (\xi_1) | \gtrsim N_0^5$. 
From (IVa), we only prove (\ref{red_BE}) in the latter case. In this case, we may assume that $\widehat{u}$ is supported on $D_2$
 and $|\tau-p_{\lambda} (\xi)| \lesssim N_0^5$. We use the Young inequality to obtain 
\begin{align*}
& N_0^{1} N_1^{-s}  \sum_{ 1 \leq M_0 \lesssim N_0^5} M_0^{-1/2+ } 
\| ( \langle \xi \rangle^s \widehat{u})*( \langle \xi \rangle^s \widehat{v}) \|_{L_{\tau,\xi}^2 (B_{M_0} )} \\
& \lesssim N_0^{1+} N_1^{-s}   
\| ( \langle \xi \rangle^s \widehat{u})*( \langle \xi \rangle^s \widehat{v}) \|_{L_{\xi}^2 L_{\tau}^{\infty} } 
\lesssim N_0^{1+} N_1^{-s} \| \langle \xi \rangle^s \widehat{u} \|_{L_{\xi}^1 L_{\tau}^2 }
\|  \langle \xi \rangle^s \widehat{v}\|_{L_{\tau,\xi}^2 } ,
\end{align*}
which is bounded by $ N_1^{-s-7/2+} \| u \|_{X^{s+s_2,b_2}} \| v \|_{ X^{s+s_2,b_2} } $ from the H\"{o}lder inequality.  
From these estimates, we only show (\ref{red_BE}) in the case both $\widehat{u}* \widehat{v}$ and $\widehat{v}$ are 
restricted to $D_1$.

\vspace{0.3em}

(IVc) Consider the case $\widehat{u}$ is supported on $D_2$. 
We use (\ref{es_L_4_22}) with $K \sim N_0$ to have
\begin{align*}
& N_0^{1} N_1^{-s}  \sum_{1 \leq M_0 \lesssim N_0^5} M_0^{-1/2+} 
\| ( \langle \xi \rangle^s \widehat{u})*( \langle \xi \rangle^s \widehat{v}) \|_{L_{\tau,\xi}^2 (B_{M_0} )} \\
& \lesssim N_0^{-1/2+} N_1^{-s-1/2}   
\|  \langle \xi \rangle^s \widehat{u} \|_{L_{\tau,\xi}^2}  \| v \|_{X_{(2,1)}^{s,1/2}}
\lesssim N_0^{-1/2+} N_1^{-s-3} \| u \|_{X^{s+s_2, b_2} } \| v \|_{X_{(2,1)}^{s,1/2}}. 
\end{align*}
In the case $M_{max}= |\tau_1-p_{\lambda} (\xi_1)| $, we immediately obtain the desired estimate because 
$\widehat{u}$ is supported on $D_2$ from $|\tau_1-p_{\lambda} (\xi_1)| \gtrsim N_1 N_0^4 \gg N_1^5$. 
Therefore we may assume that $\widehat{u}$ is restricted to $D_1$. 

\vspace{0.3em}
 
(IVd) Consider the case $M_{max} = |\tau-p_{\lambda} (\xi) | $. We use (\ref{es_L_4_12}) with $K \sim N_0$ to have
\begin{align*}
& N_0^{1} N_1^{-s}  \sum_{M_0 \gtrsim N_0^4 N_1} M_0^{-1/2} 
\| ( \langle \xi \rangle^s \widehat{u})*( \langle \xi \rangle^s \widehat{v}) \|_{L_{\tau,\xi}^2 (B_{M_0} )} \\
& \lesssim N_0^{-1} N_1^{-s-1/2}   
\| ( \langle \xi \rangle^s \widehat{u})*( \langle \xi \rangle^s \widehat{v}) \|_{L_{\tau,\xi}^2} 
\lesssim N_0^{-3} N_1^{-s-1/2} \| u \|_{X_{(2,1)}^{s,1/2}} \| v \|_{X_{(2,1)}^{s,1/2}}.
\end{align*}
The case $M_{max}=| \tau_2-p_{\lambda}(\xi_2)|$ is almost identical to this case. 

\vspace{0.3em}

{\bf Estimate for (vi).} We now prove (\ref{red_BE}) by using Lemmas~\ref{lem_L_4_1}, \ref{lem_L_4_2} and $M_{max} \gtrsim N_0^5$. 

\vspace{0.3em}

(Va) Consider the case $M_{max} =|\tau-p_{\lambda} (\xi)|$. 
We use (\ref{es_L_4_11}) with $b=b'=7/16$ to have 
\begin{align*}
& N_0^{-s+1} \sum_{M_0 \gtrsim N_0^5} M_0^{-1/2} 
\| (\langle \xi \rangle^s \widehat{u})* ( \langle \xi \rangle^s \widehat{v}) \|_{L_{\tau,\xi}^2(B_{M_0}) } \\
& \lesssim N_0^{-s-3/2} 
\| (\langle \xi \rangle^s \widehat{u})* ( \langle \xi \rangle^s \widehat{v}) \|_{L_{\tau,\xi}^2 } 
\lesssim N_0^{-s-9/4} \| u \|_{X_{(2,1)}^{s,7/16}} \| v \|_{X_{(2,1)}^{s,7/16}},
\end{align*}
which is an appropriate bound because $7/16 < b_2=79/168$.

\vspace{0.3em}

(Vb) Consider the case $M_{max} = |\tau_1-p_{\lambda}(\xi_1)|$. In this case, $\widehat{u}$ is 
supported on $D_2$. We use (\ref{es_L_4_21}) with $b=b'=7/16$ to obtain
\begin{align*}
& N_0^{-s+1} \sum_{M_0 \geq 1} M_0^{-1/2} 
\| (\langle \xi \rangle^s \widehat{u})* ( \langle \xi \rangle^s \widehat{v}) \|_{L_{\tau,\xi}^2(B_{M_0}) } \\
& \lesssim N_0^{-s+1/4} \sum_{M_0 \geq 1} M_0^{-1/16}  
\| \langle \xi \rangle^s \widehat{u} \|_{L_{\tau,\xi}^2 } \| v \|_{X_{(2,1)}^{ s,7/16} } 
\lesssim N_0^{-s-9/4} \| u \|_{X^{s+s_2,b_2} } \| v \|_{X_{(2,1)}^{s,7/16}},
\end{align*}
which shows the required estimate. 
\end{proof}

Combining Propositions~\ref{prop_linear1}, \ref{prop_linear2} and \ref{prop_BE_W}, 
the iteration argument works to construct a local-in-time solution. 
Here $B_r(\mathcal{X})$ for a Banach space $\mathcal{X}$ is defined as 
\begin{align*}
B_r(\mathcal{X}) := \bigl\{ u \in \mathcal{X}~; ~\| u \|_{\mathcal{X}} \leq r \bigr\}. 
\end{align*}
We obtain the local well-posedness for (\ref{Ka}) in the following sense. 

\begin{prop} \label{prop_LWP}
Let $s \geq -2$ and $r \geq 1$. For any $u_0 \in B_r (H^s)$, there exists $T=T(r)>0$ such that the 
unique solution $u \in W^s([0,T])$ satisfies the integral form of (\ref{Ka}) as follows;
\begin{align*}
u(t)=U_1 (t) u_0- \int_0^t U_1 (t-t') \p_x (u^2 (t')) dt'
\end{align*}
Moreover, the data-to-solution map, $H^s \ni u_0 \mapsto u \in W^s([0,T])$, is locally Lipschitz continuous. 
\end{prop}
For the details of the proof, see \cite{TK}. 
Remark that the solution $u$ obtained above satisfies $\| u \|_{W^s ([0,T])} \leq C \| u_0 \|_{H^s}$ 
for some constant $C>0$.

\section{Global well-posedness}

In this section, 
we extend the local-in-time solution obtained in Proposition~\ref{prop_LWP} to global one by the I-method. 
We use the modified energy $E_I^{(4)} (u)$ by adding two suitable correction terms to 
$E_I^{(2)} (u) (t)=\| I u(t) \|_{L^2}^2$, which is introduced by Colliander, Keel, Staffilani, Takaoka and 
Tao \cite{CoKeSt}. They established the almost conservation law for this functional to obtain GWP 
for the KdV equation when $s>-3/4$. 
Before the definition of the modified energies, we state some notations. 
Let $M: \mathbb{R}^k \rightarrow \mathbb{C}$. We say a $k$-multiplier 
$M$ is symmetry if $M(\xi_1, \xi_2,\cdots, \xi_k)= M ( \xi_{\sigma(1)}, \xi_{\sigma(2)}, \cdots ,\xi_{\sigma(k)})$ for all 
$\sigma \in S_k$. 
$[M]_{sym}$ denotes  
\begin{align*}
[M]_{sym} (\xi_1, \xi_2, \cdots, \xi_k) := \frac{1}{ k!} \sum_{\sigma \in S_k} M( \xi_{\sigma(1)}, \xi_{\sigma ( 2) }, \cdots ,\xi_{\sigma(k)}).
\end{align*}
We define a $k$-linear functional $\Lambda_k$ associated to the function $M$ acting on $k$ functions $u_1,u_2, \cdots, u_k,$
\begin{align*}
\Lambda_k (M : u_1, u_2, \cdots, u_k ):= \int_{ \xi_1 +\cdots +\xi_k=0 } M(\xi_1, \xi_2, \cdots,\xi_k ) \prod_{l=1}^{k} \widehat{u}_l (\xi_l).
\end{align*}
$\Lambda_k (M; u, \cdots ,u)$ is simply written as $\Lambda_k (M)$. 

We now define new modified energies by adding some correction terms to the original modified energy $E_I^{(2)} (u)$. 
Following $u$ is real valued and $m$ is even, we have 
\begin{align*}
E_I^{(2)} (u) (t)=\int_{\xi_1+\xi_2=0} m(\xi_1) m(\xi_2) \prod_{k=1}^2 \widehat{u}(\xi_i)
=\Lambda_2 (m(\xi_1) m(\xi_2) ).  
\end{align*}
We denote 
\begin{align*}
a_k:= i \sum_{l=1}^k \xi_l^5, \hspace{0.3cm} b_k:= i \sum_{l=1}^k \xi_l^3.
\end{align*}
We compute the time derivative of $E_I^{(2)} (u) (t)$ to obtain 
\begin{align*}
\frac{d}{ d t} E_I^{(2)} (u) (t)=\Lambda_2 ( (a_2+ \beta \lambda^{-2} b_2 )m(\xi_1) m(\xi_2)) (t)+
\Lambda_3 ( -2 i [ m(\xi_1) m(\xi_{23}) \xi_{23} ]_{sym} ) (t), 
\end{align*}
where $\xi_{ij}:= \xi_i+\xi_j$ for $i \neq j$. Note that the quadratic term vanishes because $a_2=0$ and $b_2=0$. 
So the time derivative of $E_I^{(2)} (u)$ has the cubic form as follows;  
\begin{align*}
\frac{d }{ dt} E_{I}^{(2)} (u) (t)= \Lambda_3 (M_3) (t), \hspace{0.3cm}
M_{3} (\xi_1, \xi_2, \xi_3):= -2i [ m(\xi_1) m(\xi_{23}) \xi_{23} ]_{sym}. 
\end{align*} 
We add a correction term $\Lambda_3(\sigma_3)$ 
to the modified energy $E_I^{(2)} (u) $ to construct a new modified energy $E_I^{(3)} (u)$. 
Namely, 
\begin{align*}
E_I^{(3)} (u)(t):= E_I^{(2)} (u)(t) +\Lambda_3 (\sigma_3) (t),
\end{align*} 
where the symmetric function $\sigma_3$ is determined later. Similar to above,  
the time derivative of $E_I^{(3)} (u)(t)$ is expressed by
\begin{align*}
\frac{d}{ dt }E_I^{(3)} (u)(t)=& \Lambda_3 (M_3) (t)+ \Lambda_3 ( (a_3+ \beta \lambda^{-2} b_3 ) \sigma_3 ) (t) \\
+ & \Lambda_4 (-3 i [\sigma_{3} (\xi_1,\xi_2, \xi_{34}) \xi_{34} ]_{sym} ) (t).
\end{align*}
We choose $\sigma_3= -M_3/(a_3+ \beta \lambda^{-2} b_3 )$ to cancel the cubic terms. Therefore we have 
\begin{align*}
\frac{ d } { dt }E_I^{(3)} (u)(t)= \Lambda_4 (M_4) (t), ~~ 
M_{4} (\xi_1, \xi_2, \xi_3, \xi_4)= -3i [\sigma_3 (\xi_1,\xi_2, \xi_{34}) \xi_{34} ]_{sym}.  
\end{align*}
In the same manner, we define the third modified energy as
\begin{align*}
E_I^{(4)} (u) (t):=E_I^{(3)} (u)(t)+ \Lambda_{4} (\sigma_4) (t), \hspace{0.3em}
\sigma_4:= -M_4/(a_4+  \lambda^{-2} \beta b_4).  
\end{align*}
Then, 
\begin{align*}
\frac{ d } { dt }E_I^{(4)} (u)(t)= & \Lambda_5 (M_5) (t), \\ 
M_{5} (\xi_1, \xi_2, \xi_3, \xi_4, \xi_5 ) = & -4i [\sigma_4 (\xi_1,\xi_2, \xi_3, \xi_{45}) \xi_{45} ]_{sym}.
\end{align*}
Our almost conservation law for the modified energy $E_I^{(4)} (u)$ defined above is stated as follows.

\begin{prop} \label{prop_main}
Let $0>s \geq -38/21$. Then we have
\begin{align} \label{ACL2}
\bigl| E_I^{(4)} (u) (t)-E_I^{(4)} (t_0) \bigr| \lesssim 
N^{5s} \| I u \|_{W^s([t_0-1,t_0+1])}^5,
\end{align}
for any $t_0 \in \mathbb{R}$ and $t \in [t_0-1,t_0+1]$. 
\end{prop}

This estimate implies that the growth of the modified energy $E_I^{(4)} (u)$ in time is sufficiently small. 
Combining this estimate and the fixed time difference stated below, the standard 
argument of the I-method works to establish GWP.  

We prepare some lemmas to prove the above proposition. 
Chen and Guo \cite{CG} used the mean value theorem to obtain 
the upper bound of $M_4$ as follows. 

\begin{lem} \label{lem_M_4}
Let $|\xi_1| \geq |\xi_2| \geq |\xi_3| \geq |\xi_4|$. Then we have
\begin{align} \label{M_4}
|M_4 (\xi_1,\xi_2, \xi_3, \xi_4) | \lesssim \frac{ |a_4+ \beta \lambda^{-2} b_4| m^2 (\xi_4^{*}) }
{(N+|\xi_1|)^2 (N+|\xi_2|)^2 (N+|\xi_3|)^3 (N+|\xi_4|) }
\end{align}
where $\xi_4^{*}= \min \{ |\xi_4|, |\xi_{ij}| \}$. 
\end{lem}
Compared to the KdV equation, the Kawahara equation has less symmetries. 
So it is hard to obtain this upper bound.  
Next we recall some well-known estimates for the evolution operator $e^{t (\p_x^5 - \lambda^{-2} \beta \p_x^3)}$ as follows. 

\begin{lem} \label{lem_smooth}
Let $N$ be some large dyadic number and $I$ be some time interval satisfying $| I | \lesssim 1 $. Then we have
\begin{align} \label{smooth_0}
\| u_N \|_{L_x^{\infty}L_t^{2} } & \lesssim N^{-2} \| u_N \|_{X_{(2,1)}^{0,1/2}}, \\
\label{smooth_4}
\| u_N \|_{L_x^{4} L_{t}^{\infty}} & \lesssim N^{1/4} \| u_N \|_{X_{(2,1)}^{0,1/2}},  \\
\label{smooth_2}
\| u_N \|_{L_x^{2} L_{t \in I}^{\infty}} & \lesssim N^{5/4} \| u_N \|_{X_{(2,1)}^{0,1/2}}, 
\end{align}
where $u_N= P_{ \{ |\xi| \sim N \}} u$. 
\end{lem}
For the proof, see \cite{KPV91} and \cite{KPV93}.
Combining Lemmas~\ref{lem_M_4} and \ref{lem_smooth}, 
Chen and Guo \cite{CG} showed the almost conservation law for $E_I^{(4)} (u)$ 
when $s \geq -7/4$. In fact, the use of $X_{(2,1)}^{s,1/2}$ enables us to 
gain two derivatives by the smoothing effect. 
From this, we obtain the almost conservation law for $s \geq -37/20$ by using this function space.   
But (\ref{BE}) in $X_{(2,1)}^{s,1/2}$ fails for $s<-7/4$.  
When we use the function space $W^s$, two derivatives cannot be recovered by the smoothing effects.  
So it is difficult to obtain the almost conservation law when $s<-7/4$. 
To overcome this difficulty, we establish the improved bilinear estimate which is $L_{t,x}^4$ type Strichartz estimate. 
\begin{lem} \label{lem_BR}
Let $N_1, N_2$ be dyadic numbers such that $N_1 \gg  N_2$ and $N_1 \geq 1$. For $0<b \leq 1/2$, we have
\begin{align} \label{BR}
\| u_{N_1} v_{N_2} \|_{L_{t,x}^2} \leq C   N_1^{-4b}  N_2^{1/2-b}
\| u_{N_1} \|_{X_{(2,1)}^{0,1/2} } \| v_{N_2} \|_{X_{(2,1)} ^{0,b}}.
\end{align}
\end{lem} 
This estimate explicitly implies how many derivatives are recovered in the function space $W^s$. 
This plays a crucial role to prove Proposition~\ref{prop_main}.

\begin{proof}[Proof of Lemma~\ref{lem_BR}]
Under this assumption, $|\xi| \sim |\xi_1| \sim N_1$ and $|\xi_2| \sim N_2$. 
From the algebraic relation (\ref{alg}), $M_{max} \gtrsim N_1^4 N_2$. 
In the case $N_2 \lesssim N_1^{-4}$, we easily obtain the desired estimate from 
H\"{o}lder's and Young's inequalities. So we may assume $N_1^4 N_2 \gtrsim 1$. 
Firstly, we consider the case $\langle \tau_2-p_{\lambda} (\xi_2) \rangle \lesssim N_1^4 N_2$. 
From $ \langle \tau_2-p_{\lambda} (\xi_2) \rangle^{1/2-b } \lesssim N_1^{2-4 b } N_2^{1/2-b } $, 
we use (\ref{es_L_4_22}) with $K \sim N_1$ to obtain 
\begin{align*}
\| u_{N_1} v_{N_2} \|_{L_{t,x}^2} \lesssim & 
N_1^{2-4b} N_2^{1/2-b} \| \widehat{u}_{N_1} * ( \langle \tau-p_{\lambda} (\xi) \rangle^{-1/2+b} \widehat{v}_{N_2}) \|_{L_{\tau,\xi}^2 } \\ 
\lesssim &  N_1^{-4 b}  N_2^{1/2-b } 
\| u_{N_1} \|_{X_{(2,1)}^{0,1/2}} \| v_{N_2} \|_{X_{(2,1)}^{0,b}}. 
\end{align*} 
On the other hand, in the case $ \langle \tau_2-p_{\lambda} (\xi_2) \rangle \gtrsim N_1^{4} N_2$, 
we use the H\"{o}lder inequality and the Young inequality to have 
\begin{align*}
\| u_{N_1} v_{N_2} \|_{L_{t,x}^2} \lesssim & \| \widehat{u}_{N_1} \|_{L_{\xi}^2 L_{\tau}^1}
\| \widehat{ v}_{N_2} \|_{L_{\xi}^1 L_{\tau}^2} \\
\lesssim &  N_1^{-4b } N_2^{1/2-b } \| u_{N_1} \|_{X_{(2,1)}^{0,1/2}}
\| v_{N_2} \|_{X^{0,b}}. 
\end{align*}  
\end{proof}

\begin{proof}[Proof of Proposition~\ref{prop_main}]
We may assume $t_0=0$ and $\widehat{u}$ is non-negative. Since 
\begin{align*}
|E_I^{(4)} (u) (t) -E_I^{(4)} (u) (0)| \lesssim \int_{-1}^{1} \Lambda_5 (M_5)(t) dt,
\end{align*}
for any $t \in [-1,1]$, it suffices to show that 
\begin{align} \label{ACL3}
\int_{-1}^1 \Lambda_5 \Bigl( \frac{ M_5 (\xi_1,\xi_2,\xi_3,\xi_4,\xi_5) }{m(\xi_1) m(\xi_2) m(\xi_3)m(\xi_4) m(\xi_5) } \Bigr) (t) dt
\lesssim N^{5s} \| u \|_{W^0([-1,1])}^5.
\end{align}
We suppose that $|\xi_1| \geq |\xi_2| \geq |\xi_3| \geq |\xi_4| \geq |\xi_5|$ without loss of generality. 
If $|\xi_i| \ll N$ for all $i=1,2,3,4,5$, then $M_5$ vanishes. 
So we can assume $|\xi_1| \sim |\xi_2| \gtrsim N$. Note that 
\begin{align*}
|M_5 (\xi_1, \xi_2, \xi_3, \xi_4, \xi_5) | \lesssim |\sigma_4 (\xi_3, \xi_4, \xi_5, \xi_{12}) \xi_{12}|.
\end{align*}
From $\xi_3+ \xi_4+\xi_5+\xi_{12}=0$, we only consider two cases as follows; 
\begin{align*}
\Omega_1&:= \bigl\{ ( \vec{\tau} , \vec{\xi} ) \in \mathbb{R}^5 \times \mathbb{R}^5~; ~
|\xi_1| \sim |\xi_2| \gtrsim |\xi_3| \sim |\xi_{12}| \gtrsim |\xi_4| \geq |\xi_5| \bigr\}, \\
\Omega_2 &:= \bigl\{ (\vec{\tau} , \vec{\xi} ) \in \mathbb{R}^5 \times \mathbb{R}^5~; ~
|\xi_1| \sim |\xi_2| \gtrsim |\xi_3| \sim |\xi_4| \gtrsim \max\{ |\xi_{12}| ,|\xi_5| \} \bigr\}.
\end{align*}
where $\vec{\tau}=(\tau_1,\tau_2, \cdots , \tau_5)$ and $\vec{\xi}= (\xi_1,\xi_2, \cdots ,\xi_5)$. 
In both $\Omega_1$ and $\Omega_2$, we may assume $|\xi_3| \gtrsim N$ and 
the left hand side of (\ref{ACL3}) is bounded by  
\begin{align*}
N^{5s} \int_{-1}^1 \Lambda_5 \Bigl( \langle \xi_1 \rangle^{-s} \langle \xi_2 \rangle^{-s} 
\langle \xi_3 \rangle^{-s-3} \langle \xi_4 \rangle^{-s-1} \langle \xi_5 \rangle^{-s-3} \Bigr) (t) dt,  
\end{align*}
from the upper bound of $M_4$ (\ref{M_4}). 
We use the dyadic decompositions so that the above is equivalent to 
\begin{align*}
N^{5s} \sum_{N_1} \sum_{N_2 \sim N_1} \sum_{N_3 \leq N_2} \sum_{N_4 \leq N_3} \sum_{N_5 \leq N_4} 
N_1^{-s} N_2^{-s} N_3^{-s-3} \langle N_4 \rangle^{-s-1} \langle N_5 \rangle^{-s-3} 
\| \prod_{i=1}^5 u_{N_i} \|_{L_x^1 L_{t \in [-1,1]}^1 },    
\end{align*}
where dyadic numbers $N_i$ for $i=1,2,3,4,5$. 
From the Schwartz inequality, (\ref{ACL3}) is reduced to two estimates as follows.  
\begin{align} \label{es_BE}
N_1^{-s} \langle N_5 \rangle^{-s-3} \| u_{N_1} u_{N_5} \|_{L_{t,x}^2} & \lesssim N_1^{-s-38/21} \| u_{N_1} \|_{W^0} \| u_{N_5} \|_{ W^0 }, \\ 
\label{es_TR}
N_2^{-s}  N_3^{-s-3} \langle N_4 \rangle^{-s-1} \bigl\| \prod_{i=2}^4 u_{N_i}  \bigr\|_{L_{t,x}^2} & \lesssim N_2^{-s-38/21} N_3^{-s-38/21-} 
\prod_{i=2}^4 \| u_{N_i} \|_{W^0} .
\end{align}

We prove these estimates by using the H\"{o}lder inequality, the Sobolev inequality and 
the improved bilinear estimate (\ref{BR}). 
Firstly, we show (\ref{es_BE}). 

\vspace{0.3em}

(Ia) Consider the case $\widehat{u}_{N_1}$ is restricted to $D_2$. 
The H\"{o}lder inequality and the Sobolev inequality imply that 
\begin{align*}
N_1^{-s} \langle N_5 \rangle^{-s-3} \| u_{N_1} u_{N_5} \|_{L_{t,x}^2} \lesssim 
& N_1^{-s} \langle N_5 \rangle^{-s-3} \| u_{N_1} \|_{L_{t,x}^2} \| u_{N_5} \|_{L_{t,x}^{\infty} } \\
\lesssim & N_1^{-s-5/2} \langle N_5 \rangle^{-s-5/2} \| u_{N_1} \|_{X^{s_2, b_2}} \| u_{N_5} \|_{Y^0}, 
\end{align*}
which is an appropriate bound where $s_2=25/168$ and $b_2=79/168$. 
So we only prove (\ref{es_BE}) in the case $\widehat{u}_{N_1}$ is restricted to $D_1$.

\vspace{0.3em}

(Ib) Consider the case $N_1 \sim N_5 \gtrsim N$.  
We combine (\ref{smooth_0})--(\ref{smooth_2}) to obtain the $L_{t,x}^4$ estimate, 
 $\| u_N \|_{L_{t,x}^4} \lesssim N^{-3/8} \| u_N \|_{X_{(2,1)}^{0,3/8}}$, for $N \gg 1$.  
We use this estimate and the H\"{o}lder inequality to obtain 
\begin{align*}
N_1^{-2s-3} \| u_{N_1} u_{N_5} \|_{L_{t,x}^2} & \lesssim N_1^{-2s-3} \| u_{N_1} \|_{L_{t,x}^4} \| u_{N_5} \|_{L_{t,x}^4} \\
& \lesssim 
N_1^{-2s-15/4} \| u_{N_1} \|_{X_{(2,1)}^{0,3/8}} \| u_{N_5} \|_{X_{(2,1)}^{0,3/8}},
\end{align*}
which implies the desired estimate. So we may assume $N_1 \gg N_5$.   
Then we use (\ref{BR}) with $b=19/42$ to have 
\begin{align*}
N_1^{-s} \langle N_5 \rangle^{-s-3} \| u_{N_1} u_{N_5} \|_{L_{t,x}^2} 
\lesssim N_1^{-s-38/21} \langle N_5 \rangle^{-s-62/21} \| u_{N_1} \|_{X_{(2,1)}^{0,1/2}} 
\| u_{N_5} \|_{X_{(2,1)}^{0,19/42} },
\end{align*}
which is an appropriate bound.

\vspace{0.3em}

Secondly, we prove (\ref{es_TR}).

(IIa) Consider the case $\widehat{u}_{N_2}$ is supported on $D_2$. 
We use the H\"{o}lder inequality and  the 
Sobolev inequality to obtain 
\begin{align*}
& N_2^{-s} N_3^{-s-3} \langle N_4 \rangle^{-s-1} \| \prod_{i=2}^4 u_{N_i} \|_{L_{t,x}^2} 
\lesssim N_2^{-s} N_3^{-s-3}  \langle N_4 \rangle^{-s-1} 
\| u_{N_2} \|_{ L_{t,x}^2} \| u_{N_3} \|_{ L_{t,x}^{\infty} }  \| u_{N_4} \|_{ L_{t,x}^{\infty} } \\
& \hspace{0.3cm} \lesssim N_2^{-s-5/2} N_3^{-s-5/2} \langle N_4 \rangle^{-s-1/2}
\| u_{N_2} \|_{X^{ s_2,b_2 } } \| u_{N_3} \|_{Y^0} \| u_{N_4} \|_{Y^0},
\end{align*}
which implies the desired estimate. 
So we only estimate (\ref{es_TR}) in the case $\widehat{u}_{N_2}$ is supported on $D_1$. 

\vspace{0.3em}

(IIb) Consider the case $N_2 \sim N_4 \gtrsim N$. If there exists at least one of 
$i=2,3,4$ such that $\widehat{u}_{N_i}$ is supported on $D_2$, then we immediately 
obtain the desired estimate following the above estimate. 
Therefore we may assume that $\widehat{u}_{N_i}$ is restricted to $D_1$ for all $i=2,3,4$. 
In this case, we use (\ref{smooth_0}), (\ref{smooth_4}) and the H\"{o}lder inequality 
to obtain
\begin{align*}
N_2^{-3s-4} \| \prod_{i=2}^4 u_{N_i} \|_{L_{t,x}^2} &\lesssim N_2^{-3s-4} \| u_{N_2} \|_{L_x^{\infty} L_t^2} 
\| u_{N_3 } \|_{L_x^4 L_t^{\infty}} \| u_{N_4} \|_{L_x^4 L_{t}^{\infty} } \\
& \lesssim N_2^{-3s-11/2}  \prod_{i=2}^4 \| u_{N_i} \|_{X_{(2,1)}^{0,1/2}}. 
\end{align*}

(IIc) Consider the case $N_2 \gg N_4$. We first deal with the case $N_4 \leq 1$. 
We use H\"{o}lder's and Sobolev's inequalities to have
\begin{align*}
N_2^{-s} N_3^{-s-3} 
\| \prod_{i=2}^4 u_{N_i} \|_{L_{t,x}^2} & \lesssim 
N_2^{-s} N_3^{-s-3}  \| u_{N_2} u_{N_4} \|_{L_{t,x}^2} \| u_{N_3} \|_{L_{t,x}^{\infty} } \\
&  \lesssim N_2^{-s} N_3^{-s-5/2} \| u_{N_3} \|_{Y^0 } 
\| u_{N_2} u_{N_4} \|_{L_{t,x}^2 },
\end{align*}
which is bounded by 
\begin{align*}
N_2^{-s-38/21} N_3^{-s-5/2} \| u_{N_2} \|_{X_{(2,1)}^{0,1/2} } \| u_{N_3} \|_{Y^0} 
\| u_{N_4} \|_{X_{(2,1)}^{0,19/42}}
\end{align*}
from (\ref{BR}) with $b=19/42$. Next, 
we prove (\ref{es_TR}) when $N_4 \geq 1$. 
We first estimate in the case $\widehat{u}_{N_4}$ is supported on $D_2$. 
The H\"{o}lder inequality and  the Sobolev inequality imply that 
\begin{align*}
N_2^{-s} N_3^{-s-3} \langle N_4 \rangle^{-s-1} 
\| \prod_{i=2}^4 u_{N_i} \|_{L_{t,x}^2} \lesssim N_2^{-s} N_3^{-s-5/2} \langle N_4 \rangle^{-s-1} \| u_{N_2} u_{N_4} \|_{L_{t,x}^2} \| u_{N_3} \|_{Y^0},
\end{align*}
which is bounded by 
\begin{align*}
N_2^{-s-4b_2} N_3^{-s-5/2} \langle N_4 \rangle^{-s-47/42} 
\| u_{N_2} \|_{X_{(2,1)}^{0,1/2}} \| u_{N_3} \|_{Y^0} \| u_{N_4} \|_{X_{(2,1)}^{s_2,b_2}}
\end{align*} 
from (\ref{BR}) with $b=b_2=79/168$. Next, we consider the case 
$\widehat{u}_{N_4}$ is supported on $D_1$. We use H\"{o}lder's and Sobolev's inequalities 
and (\ref{BR}) with b=1/2 to have 
\begin{align*}  
& N_2^{-s} N_3^{-s-3} \langle N_4 \rangle^{-s-1} 
\| \prod_{i=2}^4 u_{N_i} \|_{L_{t,x}^2} 
\lesssim N_2^{-s} N_3^{-s-5/2} \langle N_4 \rangle^{-s-1} \| u_{N_2} u_{N_4} \|_{L_{t,x}^2} \| u_{N_3} \|_{Y^0} \\
 \hspace{0.3cm} & \lesssim N_2^{-s-2} N_3^{-s-5/2} \langle N_4 \rangle^{-s-1}
\| u_{N_2} \|_{X_{(2,1)}^{0,1/2}} \| u_{N_3} \|_{Y^0} \| u_{N_4} \|_{X_{(2,1)}^{0,1/2}},
\end{align*}
which is an appropriate bound. 
\end{proof}

\vspace{0.3em}

\noindent
{\bf Remark.} We add a suitable correction term to $E_I^{(4)} (u)$ 
to construct a new modified energy $E_I^{(5)} (u)$.  
If we use the modified energy $E_I^{(5)} (u)$, 
we probably obtain the almost conservation law in the same regularity $s=-38/12$. 
This is a reason why we do not expect to gain more $4b_1$ derivatives by smoothing effects 
when the norm in $D_3$ is defined as $\|  \cdot \|_{X_{(2,1)}^{s,b_1}}$ and any derivatives from $M_5 $ bounds. 

\vspace{0.5em}

Next, we estimate the difference between the modified energies $E_I^{(2)}(u) $ and $E_I^{(4)}(u) $ when time is fixed. 
\begin{prop} \label{prop_FTD}
Let $0> s \geq -2$. Then there exists $C>0$ such that 
\begin{align} \label{FTD}
| E_I^{(4)}(u) (t_0) - E_I^{(2)}(u) (t_0) | \leq C ( \| I u(t_0) \|_{L^2}^3+ \| I u(t_0 ) \|_{L^2}^4 ),
\end{align}
for any $t_0 \in \mathbb{R}$.  
\end{prop}
We call this estimate the fixed point difference.  Compared to the argument of \cite{CoKeSt}, we need to estimate more sharply 
in order to obtain the above estimate when $-2 \leq s <-7/4$ 
\begin{proof}
We may assume that $\widehat{u}$ is non-negative. 
From the definition of the modified energy, it suffices to show  
\begin{align*}
|\Lambda_3 (\sigma_3) (t_0) |  \lesssim \| I u (t_0) \|_{L^2}^3,
\hspace{0.3cm} 
| \Lambda_{4} (\sigma_4)( t_0) |  \lesssim \| I u (t_0) \|_{L^2}^4. 
\end{align*} 
Note that the mean value theorem shows the following $M_3$ bounds as follows;
\begin{align*}
|M_3 (\xi_1, \xi_2, \xi_3) | \lesssim |\xi_3| m^2 (\xi_3),
\end{align*}
where $|\xi_1| \geq |\xi_2| \geq |\xi_3|$. 
From $M_3$ bounds and $M_4$ bounds, (\ref{FTD}) is reduced to the following estimates. 
\begin{align}
\label{sigma_3}
\Bigl| \Lambda_3  \Bigl( \frac{  m (\xi_3)}{\prod_{i=1}^2 (N+|\xi_i| )^2 m( \xi_i) } \Bigr) (t_0) 
\Bigr| \lesssim &  \| u (t_0) \|_{L^2}^3, \\
\label{sigma_4}
\Bigl| \Lambda_4  \Bigl( \frac{ m^{2}(\xi_4^{*}) }{ \prod_{i=1}^4 (N+ |\xi_i| )^2 m(\xi_i) } \Bigr) (t_0) 
\Bigr| \lesssim & \| u (t_0) \|_{L^2}^4.
\end{align}
where $\xi_4^{*} := \min \{ |\xi_i|, |\xi_{jk}| \}$. 
We first prove (\ref{sigma_3}). We only consider $|\xi_1| \sim |\xi_2| \gtrsim N$ since $M_3$ vanishes in the other cases. 
Combining the H\"{o}lder inequality and the Sobolev inequality,  the left hand side of (\ref{sigma_3}) is bounded by  
\begin{align*}
N^{2s}  \Bigl| \int [\langle \p_x \rangle^{-s-2} u (t_0) ]^2 I u(t_0) dx \Bigr|  & \lesssim 
N^{2s} \| \langle \p_x \rangle^{-s-2} u(t_0) \|_{L^4}^2 \| I u(t_0) \|_{L^2} \\
& \lesssim N^{2s} \| u(t_0) \|_{L^2}^3, 
\end{align*}
when $0>s \geq -7/4 $. 
Moreover, in the case $-7/4 > s \geq  -2$, the Sobolev and the Hausdorff-Young inequalities imply 
\begin{align*}
N^{2s} \Bigl| \int [\langle \p_x \rangle^{-s-2} u(t_0) ]^2 I u(t_0) d x  \Bigr| \lesssim & 
N^{2s} \| \langle \p_x \rangle^{-s-2} u(t_0) \|_{L^{2/(-2s-3)}}^2 \| I u(t_0) \|_{L^{1/(2s+4)}} \\
\lesssim & N^{2s} \| u(t_0) \|_{L^2}^2 \| m \widehat{u} (t_0) \|_{L_{\xi}^{1/(-2s-3)}}.
\end{align*} 
Here we use the H\"{o}lder inequality to have
\begin{align*}
\| m \widehat{u} (t_0) \|_{L_{\xi}^{1/(-2s-3)}} \lesssim   \| m \|_{L_{\xi}^{2/(-4s-7)}} \| \widehat{u} (t_0) \|_{L_{\xi}^2} 
\lesssim  N^{-2s-7/2} \| u (t_0) \|_{L^2}, 
\end{align*}
which shows that the left hand side of  (\ref{sigma_3}) is bounded by $N^{-7/2} \| u(t_0) \|_{L^2}^3$.

Next we show (\ref{sigma_4}). We assume $|\xi_1| \geq |\xi_2| \geq |\xi_3| \geq |\xi_4|$ without loss of generality. 
From $M_4$ bounds and Sobolev's inequality, 
the left hand side of (\ref{sigma_4}) is bounded by 
\begin{align*}
N^{4s}  \| \langle \p_x \rangle^{-s-2} u (t_0) \|_{L^4}^4  
\lesssim & N^{4s} \| u (t_0) \|_{L^2}^4
\end{align*}
when $0> s \geq -7/4$. On the other hand, we consider the case $-7/4 >  s \geq -2$. 
In the case $\xi_4^{*}= | \xi_{ij} |$, from $M_4$ bounds, it suffices to show that 
\begin{align*}
N^{4s} \| I [ \langle \p_x \rangle^{-s-2} u (t_0)]^2 \|_{L^2}^2 \lesssim \| u(t_0) \|_{L^2}^4.  
\end{align*}
We use the H\"{o}lder inequality and the Young inequality to have 
\begin{align*}
\| I [\langle \p_x \rangle^{-s-2} u(t_0) ]^2 \|_{L^2} 
= & \| m(\xi) (\langle \xi \rangle^{-s-2} \widehat{u} (t_0) )* (\langle \xi \rangle^{-s-2} \widehat{u} (t_0) ) \|_{L_{\xi}^2} \\
\lesssim 
& \| m \|_{L_{\xi}^2 } \| \langle \xi \rangle^{-s-2} \widehat{u}(t_0) \|_{L_{\xi}^2}^2 
\lesssim N^{1/2} \| u (t_0 ) \|_{L^2}^2,
\end{align*}
which shows the desired estimate. 
Next, we deal with the case $\xi_4^{*}=|\xi_4| $. 
It suffices to show that 
\begin{align*}
N^{3s} \Bigl| \int [\langle \p_x \rangle^{-s-2} u (t_0)]^2 ( \langle \p_x \rangle^{-s-3} u (t_0) ) (\langle \p_x \rangle^{-1} I u(t_0) ) dx
\Bigr| 
\lesssim \| u(t_0) \|_{L^2}^4.
\end{align*}
The Sobolev inequality and 
the H\"{o}lder inequality inequality imply that  
\begin{align*}
& N^{3s} \| \langle \p_x \rangle^{-s-2} u(t_0) \|_{L^2}^2 \| \langle \p_x \rangle^{-s-3} u(t_0) \|_{L^{\infty}}
\| \langle \p_x \rangle^{-1} I u(t_0) \|_{L^{\infty}}  \\
& \hspace{0.3cm} \lesssim  N^{3s} \| u(t_0 ) \|_{L^2}^2 \| \langle \p_x \rangle^{-s-5/2} u(t_0) \|_{L^2} 
\| \langle \p_x \rangle^{-1/2} I u(t_0) \|_{L^2} \lesssim  N^{3s} \| u(t_0) \|_{L^2}^4. 
\end{align*}
\end{proof}

Combining Propositions~\ref{prop_main} and \ref{prop_FTD}, 
we can find a constant $C_1>0$ such that
\begin{align*}
\sup_{-N^{-5s} \leq t \leq N^{-5s}} \| I u(t) \|_{L^2} \leq C_1 \| I u(0) \|_{ L^2 }
\end{align*}
when $0> s \geq -38/21$. For the details, see \cite{CoKeSt}.  
Following 
\begin{align*}
\| I u_{\lambda} (0) \|_{L^2} \leq C_2 \lambda^{-s-7/2} N^{-s} \| u_0 \|_{H^s},
\end{align*}
for some constant $C_2$, we take $\lambda^{-s-7/2} N^{-s} =\varepsilon_0 \ll 1$
If $ \lambda^5 T \leq N^{-5s}$, then we have 
\begin{align*}
& \sup_{-T \leq t \leq T} \| u(t) \|_{H^s} \leq \lambda^{7/2} \sup_{-\lambda^5 T \leq t \leq \lambda^5 T} \| I u_{\lambda} (t) \|_{L^2} \\
& \hspace{0.3cm} \leq   C_1 \lambda^{7/2} \| I u_{\lambda} (0 ) \|_{L^2} \leq  \varepsilon_0 C_1 C_2 \lambda^{-s} N^{-s} \| u_0 \|_{H^s}.
\end{align*}
Therefore we have the following upper bound of the growth order of $H^s$. 
\begin{align*}
\sup_{-T \leq t \leq T} \| u (t) \|_{H^s} \leq C T^{7/5 (2s+5)} \| u_0 \|_{H^s},
\end{align*}
for $-38/21 \leq s <0$.


\begin{thebibliography}{99}
\bibitem{BT}
J. Bejenaru and T. Tao, \textit{Sharp well-posedness and ill-posedness results for a quadratic nonlinear Schr\"{o}dinger equation}, J. Funct. Anal. \textbf{233} (2006), 228--259.
\bibitem{Bo} J. Bourgain, \textit{Fourier restriction phenomena for certain lattice subset applications to nonlinear evolution equation}, 
Geometric and functional Anal. 
\textbf{3} (1993), 107--156, 209--262. 
\bibitem{CG} W. Chen and Z. Guo, \textit{Global well-posedness and I method for the fifth-order Korteweg-de Vries equation}, 
to appear in J. D'Anal. Math. 
\bibitem{CLMW} 
W. Chen, J. Li, C. Miao and J. Wu, \textit{Low regularity solution of two fifth-order KdV type equations}, J. D'Anal. Math. \textbf{107} (2009), 
221--238.
\bibitem{CoKe} J. Colliander, M. Keel, G. Staffilani, H. Takaoka and T. Tao,
\textit{Global well-posedness for KdV in Sobolev spaces of negative index }, Electron. J. Differential Equations {\bf 2001}, No. 26, pp. 1--7. 
\bibitem{I02}
J. Colliander, M. Keel, G. Staffilani, H. Takaoka and T. Tao,
\textit{Almost conservation laws and global rough solutions to a nonlinear Sch\"{o}dinger equation}, 
Math. Res. Lett. {\bf 9} (2002), 659--682. 
\bibitem{CoKeSt} J. Colliander, M. Keel, G. Staffilani, H. Takaoka and T. Tao, 
\textit{Sharp global well-posedness for KdV and modified KdV on $\mathbb{R}$ and $\mathbb{T}$}, J. Amer. Math. Soc. 
\textbf{16} (2003), 705--749. 
\bibitem{CDT} S. Cui, D. Deng and S. Tao,  
\textit{Global existence of solutions for the Cauchy problem of the Kawahara equation with $L^2$ initial data}, Acta Math. Sin. {\bf 22} (2006), 
1457--1466. 
\bibitem{Gu} Z. Guo, \textit{Global well-posedness of Korteweg-de Vries equation in $H^{-3/4}$}, J. Math, Pures Appl. {\bf 91} 
(2009), 583--597.  
\bibitem{TK} T. K. Kato, \textit{Local well-posedness for Kawahara equation}, Adv. Differential Equations {\bf 16} (2011), 
no. 3--4, 257--287. 
\bibitem{TK1}
T. K. Kato, \textit{Well-posedness for the fifth order KdV equation} to appear in Fankcialaj Ekvacioj.
\bibitem{Ka}
T. Kawahara, \textit{Oscillatory solitary waves in dispersive media}, J. Phys. Soc. Japan, {\bf 33} (1972), 260--264.
\bibitem{KPV91} C. E. Kenig, G. Ponce and L. Vega, \textit{Oscillatory integrals and regularity of dispersive equations}, Indiana Univ. 
Math. J. {\bf 40} (1991), 33--69. 
\bibitem{KPV93} C. E. Kenig, G. Ponce and L. Vega, \textit{Well-posedness and scattering results for the generalized 
Korteweg-de Vries equation via the contraction principle}, Comm. Pure Appl. Math. {\bf 46} (1993), 
527--620. 
\bibitem{KPV96} C. E. Kenig, G. Ponce, and L. Vega, 
 \textit{A bilinear estimate with applications to the KdV equation}, J. Amer. Math. Soc, \textbf{9} (1996) no. 2, 573--603.
\bibitem{Ki}
N. Kishimoto, \textit{Well-podeness of the Cauchy problem for the Korteweg-de Vries equation at critical regularity}, 
Differential Integral Equations {\bf 22} (2009), 447--464. 
\bibitem{KT}
N. Kishimoto and K. Tsugawa, \textit{Local well-posedness for quadratic Schr\"{o}dinger equations and "good'' Boussinesq equation},
 Differential Integral Equations \textbf{23} (2010), no. 5--6, 463--493.
\bibitem{Ta}
T. Tao, \textit{Multiplier weighted convolution of $L^2$ functions and application to nonlinear dispersive equations}, 
Amer. J. Math. {\bf 123} (2001), 839--908. 
\bibitem{WCD}
H. Wang, S. Cui and D. Deng, \textit{Global existence of solutions for the Kawahara equation in Sobolev space of negative indices}, 
Acta. Math. Sin. {\bf 23} (2007), 1435--1446. 
\bibitem{YL} W. Yan and Y. Li, \textit{The Cauchy problem for Kawahara equation in Sobolev spaces with low regularity}, Math. Method Appl. 
Sci. {\bf 33} (2010), no. 14, 1647--1660.
\end{thebibliography}
\end{document}